\documentclass[reqno,11pt]{amsart}

\usepackage[a4paper,margin=25mm]{geometry}
\usepackage{mathtools}
\usepackage{dsfont}
\usepackage{xcolor}
\usepackage[foot]{amsaddr}
\usepackage[shortlabels]{enumitem}
\usepackage{graphicx}
    \graphicspath{{fig/}}
\usepackage{subfigure}
\usepackage[noadjust]{cite}
\usepackage{placeins}

\newtheorem{theorem}{Theorem}
\newtheorem{lemma}{Lemma}
\newtheorem{proposition}{Proposition}
\theoremstyle{definition}
\newtheorem{definition}{Definition}
\newtheorem{example}{Example}
\theoremstyle{remark}
\newtheorem{remark}{Remark}

\newcommand{\entropy}{\mathcal H}
\newcommand{\abs}[1]{\left\lvert#1\right\rvert}
\newcommand{\set}[1]{\left\{#1\right\}}
\newcommand{\real}{\mathbb R}
\newcommand{\ind}{\mathds 1}
\DeclareMathOperator{\supp}{supp}

\allowdisplaybreaks

\begin{document}
\title{Differential Shannon and R\'enyi entropies revisited}
\author[Yuliya Mishura and Kostiantyn Ralchenko]{Yuliya Mishura$^{1}$ and Kostiantyn Ralchenko$^{1,2}$}
\address{$^1$Taras Shevchenko National University of Kyiv}
 
\address{$^2$University of Vaasa}
\email{yuliyamishura@knu.ua, kostiantynralchenko@knu.ua}

\thanks{Y.M. is supported by The Swedish Foundation for Strategic Research, grant UKR24-0004, and by the Japan
Science and Technology Agency CREST, project reference number JPMJCR2115.
K.R. is supported by the Research Council of Finland, decision number 367468.}

\begin{abstract}
Shannon entropy for discrete distributions is a fundamental and widely used concept,  but its continuous analogue, known as differential entropy, lacks essential properties such as positivity and compatibility with the discrete case. In this paper, we analyze this incompatibility in detail and illustrate it through examples. To overcome these limitations, we propose modified versions of Shannon and R\'enyi entropy that retain key properties, including positivity, while remaining close to the classical forms. We also define compatible discrete functionals and study the behavior of the proposed entropies for the normal and exponential distributions.
\end{abstract}



\maketitle
\section{Introduction}\label{sec-intr}
Since C. Shannon's seminal work \cite{Shannon48}, the definition of the entropy of a discrete distribution has been used in a wide variety of applications, including information technology, physics, engineering, communications, biology, medicine, economics, finance, cryptography, machine learning and many other fields. Among the examples, we mention just \cite{ali2023shannon,carcassi2021variability,dwivedi2022application,ellerman2021new,Feutrill21,nascimento2024electron,Rodriguez22,zachary2021urban}, some other examples are contained in \cite{BMR,entropy-gaussian}, but this list by no means can be exhausted here. The success of this concept is eloquently demonstrated by the 111,935 citations to the article \cite{Shannon48} in Google Scholar.  The definition of the entropy of a discrete distribution is perfect in the sense that entropy is strictly positive for any non-degenerate distribution, it corresponds to the notion of  the Gibbs entropy in  thermodynamic theory and  it satisfies a number of  axiomatic properties that uniquely determine it \cite{Aczel74}.

However, it is well known that when moving from a discrete distribution to a continuous one, Shannon entropy loses some necessary properties. This happens for a clear and long-explained reason: as the number of events increases and probabilities ``disperse'',  entropy also increases without any restrictions. In a standard situation, as a rule, a term of the type $\log N$ appears  as $N\to \infty. $ We shall describe both standard and nonstandard rates of divergence of discrete entropy to $\infty$ in Section \ref{section-diff-incomp}.  One of the attempts to adjust the notions of discrete and continuous Shannon entropies was made by E.\,T.~Jaynes \cite{Jaynes62} by introducing limiting density of discrete points, it is also described in Section \ref{section-diff-incomp}. However, the term $\log N$ is also present in his considerations. Another possibility is the following one: having   established that continuous random
objects do not allow  existence of a finite absolute measure of uncertainty (entropy), it is however possible to introduce a relative quantitative measure of uncertainty in the continuous
case as well. As a standard for comparison, it is possible to  take the uncertainty of some simple distribution, for example, uniform in an interval of width that tends to zero, and get the entropy of continuous distribution as some relative value. For more detailed information see, e.g., \cite{strat}. The name ``differential entropy'' comes, as we understand, from the fact that in this case distribution function is, in a certain sense, differentiable. Thus, the concept of entropy of a continuous distribution is to some extent relative, but the entropy itself, if it exists, is considered and used as a fixed number, and its connection with other distributions is ignored. However, this fixed value can be either positive or negative or even zero, and for no apparent reason it can be zero for a non-degenerate distribution whose connection with some other reference distribution cannot be traced, and therefore zero entropy seems completely illogical. Therefore our idea is to propose alternative versions of Shannon  entropy and to study their properties. On the one hand, we decided not to move far from the original Shannon entropy, on the other hand, to ensure the positivity of the obtained alternative entropies. Then we apply the same approach to the  R\'enyi entropy. 

The paper is organized as follows. In Section \ref{Sec2} we give the basic definitions of Shannon entropy for discrete and continuous distributions and consider some ``bad'' example  where the discrete entropy is infinite (with the sign $+$, of course), and two examples of infinite  differential entropy, both with    signs $+$ and $-$. In Section \ref{section-diff-incomp} we prove the incompatibility of differential  Shannon  entropy   with its discrete counterparts. This result is very well known, and therefore Lemma   \ref{lemmatransit1} can be considered as the part of some survey, however, we preferred to give it a rigorous proof, as opposed to numerous physically-rigorous arguments, and supply the result with several examples.  In Section \ref{sect.4} we propose alternative versions of Shannon  entropy and study their properties. Also, we propose discrete functionals compatible with differential Shannon entropy and its alternatives. In principle, again, for the differential entropy the form of a compatible discrete functional is very well-known, and again, we supply this notion with rigorous proof. Rigorous proof needs some additional assumptions that are discussed in detail. Then the form of  compatible discrete functionals for alternatives is obvious.  In Section \ref{sec5.behav} the behavior of alternative versions of Shannon differential  entropy as the functions of parameters of distributions is studied. In Section~\ref{sec6} we go the same steps, but more briefly, for R\'enyi entropy.

\section{Standard discrete and differential Shannon entropies and some examples}\label{Sec2}
Let us recall notions of Shannon entropy for discrete and absolutely continuous distribution.

\begin{definition}
Let $\{p_k, k \ge 1\}$ be a discrete distribution (with finite or countable number of non-zero probabilities).
Then its \emph{Shannon entropy} equals
\[
\entropy_{SH} (\{p_k, k \ge 1\}) = - \sum_{k \ge 1} p_k \log p_k.
\]
\end{definition}

\begin{remark}\label{rem:1} Shannon entropy of the discrete distribution is always positive and strictly positive as far as  the distribution is non-degenerate. 
If the number of $p_k$ is countable, it is assumed that the series
$\sum_{k \ge 1} p_k \abs{\log p_k} < \infty$.
Otherwise, we say that the distribution has infinite entropy.
\end{remark}

\begin{example} As a simple  example of the distribution with the infinite entropy, consider
$L \coloneqq \sum_{k=2}^\infty \frac{1}{k \log^2 k} < \infty$
and define
\[
p_k = \frac{1}{L k \log^2 k}.
\]
Then
\[
-\log p_k = \log L + \log k + 2 \log \log k,
\]
and the series
$\sum_{k=2}^\infty p_k \abs{\log p_k} = \infty$.
\end{example}

\begin{definition}
Let $\{p(x), x \in \real\}$ be a density of a probability distribution.
Then its \emph{Shannon entropy} (sometimes called \emph{differential entropy}) equals
\begin{equation}\label{entrShan0}
\entropy_{SH} (\{p(x), x \in \real\}) = - \int_\real p(x) \log p(x)\,dx=  \int_\real p(x)  \log \frac{1}{p(x)}\,dx,
\end{equation}
if
$\int_\real p(x) \abs{\log p(x)}\,dx < \infty$.
Otherwise, we say that the distribution has infinite entropy.
\end{definition}

\begin{remark} Differential entropy can be of any sign and even zero for the non-degenerate distribution. 
Infinite differential entropy can be both $+\infty$ and $-\infty$.
\end{remark}
\begin{example} Let
\[
p(x) = \frac{\log 2}{x \log^2 x} \ind\{x \ge 2\}.
\]
Then
\[
  \int_2^\infty p(x) \log \frac{1}{p(x)} \,dx
=   \int_2^\infty \frac{\log 2}{x \log^2 x} \left( \log x + 2\log\log x-\log\log 2  \right) \,dx = +\infty.
\]
\end{example}
\begin{example} Let
\[
p(x) = L^{-1} \sum_{k=2}^\infty k \ind\set{x \in \left[k, k+\frac{1}{k^2\log^2 k}\right]},
\]
where $L$ is defined in Example  \ref{rem:1}. Then
\[
- \int_\real p(x) \log p(x) \,dx
= -L^{-1} \sum_{k=2}^\infty \frac{k \log k}{k^2 \log^2 k}
= -\infty. 
\]
\end{example}

 \section{Differential  Shannon  entropy is incompatible with its discrete counterparts}\label{section-diff-incomp}

Throughout the paper, we shall use the following notations and assumptions, in what follows referred as Assumption (A). 
\begin{itemize}
\item[(A)] Denote $p=p(x),\,x \in \real$ the density of probability distribution, $F(x)=\int_{-\infty}^x p(y) dy$ be its cumulative distribution function, $\pi_N=\{x_k^N, k = 0, \dots, k_N\}$ be a sequence of partitions of $\real$ such that
$x_0^N \to -\infty$, $x_{k_N}^N \to \infty$  as $N \to \infty$, $\Delta_k^N=x_k^N - x_{k-1}^N$ and $\Delta F_k^N=F\left(x_k^N\right) - F\left(x_{k-1}^N\right).$ Also, we assume that $|\pi_N|=\max_{1\le k \le k_N}\Delta_k^N\to 0 $ as $N \to \infty$.
\end{itemize}

As it was already mentioned, it is very well known that differential entropy is not a continuous analog of discrete Shannon entropy. In order to clarify the situation,  consider  a sequence of quite natural discretizations of a continuous distribution  and  obtain an infinite limit for the corresponding entropies. It is performed in  the following lemma. We formulate it for the distribution  with continuous density, for technical simplicity, however, it admits the generalization to arbitrary density. At the physical level of rigor, this   fact has been discussed for a very long time, but we provide here a strictly mathematical proof, which is  very simple.
\begin{lemma}\label{lemmatransit1}
Let $p(x),\,x \in \real$, be a  density of probability distribution, $p\in C(\real).$ 
Then, in terms of Assumption (A), 
\begin{align*}
\entropy_{SH}^N &\coloneqq -\sum_{k=1}^{k_N}  \Delta F_k^N \log\Delta F_k^N
- F\left(x_0^N\right) \log F\left(x_0^N\right) \\*
&\quad - \Bigl(1 - F\left(x_{k_N}^N\right)\Bigr) \log\Bigl(1 - F\left(x_{k_N}^N\right)\Bigr)
\to +\infty \quad \text{as } N \to \infty,
\end{align*}
where we put $\Delta F_k^N \log\Delta F_k^N=0$ if $\Delta F_k^N=0$, and similar assumption is made for the first and last terms. 
\end{lemma}
\begin{remark} Of course, we will obtain the same result for the simplified sum 
   $$\widetilde\entropy_{SH}^N=-\sum_{k=1}^{k_N}  \Delta F_k^N \log\Delta F_k^N, $$
   because $F\left(x_0^N\right) \log F\left(x_0^N\right)\to 0$ and $\Bigl(1 - F\left(x_{k_N}^N\right)\Bigr) \log\Bigl(1 - F\left(x_{k_N}^N\right)\Bigr)\to 0$ as $N\to \infty$. Here and in what follows we use that $x\log x\to 0$ as $x\to 0$ without mentioning it again.
\end{remark}

\begin{proof}
Note that all terms in $\entropy_{SH}^N$ are strictly positive or equal to zero (taken with their minuses, of course).  Denote $s(p)=\supp \{p(x), x\in \real\}$. Since
$\int_\real p(x) dx = 1$, it follows that 
$\lambda\{x : p(x) \ge M\} \to 0$ as $M \to \infty$, 
where $\lambda$ is the Lebesgue measure on $\real$.
Then it follows from the continuity of $p$ that we can find   some $0<m_1<M_1$ and   the interval $[a,b] \subset s(p)$ such that
  $0 < m_1 < p(x) \le M_1$ on $[a,b]$, and $a<b$. Consider  those points $x_k^N$ of partition which are inside $[a,b]$, and let $x_{k^N_1}^N$ and $x_{k^N_2}^N$ be the   left endpoint and  right endpoint  of such $x_k^N$.  Since $x_{k^N_1}^N\downarrow a$ and   $x_{k^N_2}^N \uparrow b$  as $N\to \infty$, there exists $N_0$ such that for $N>N_0$ it holds that $x_{k^N_2}^N-x_{k^N_1}^N>\frac{b-a}{2}$.
Note that for $k^N_1<k\le k^N_2$ we have the inequalities 
$0<m_1\Delta_{k}^N \le \Delta F_{k}^N \le M_1 \Delta_{k}^N\le M_1|\pi_N|\to 0$, therefore there exists $N_1$ such that for $N>N_1$
   logarithms of the increments are nonzero. Consequently, 
\[
  \log\Delta F_{k}^N 
<- \log \frac{1}{M_1 |\pi_N|}.
\] 
Then for $N>N_0\vee N_1$
\[
\entropy_{SH}^N 
\ge - \sum_{k=k^N_1+1}^{k^N_2} \Delta F_{k}^N \log\Delta F_{k}^N
> \frac{b-a}{2} m_1 \log\frac{1}{M_1 |\pi_N|},
\]
where the latter value tends to $+\infty$ as $N\to\infty$.
Lemma is proved.
\end{proof} 
\begin{remark} With the same success, in the course of the proof we could consider not the intervals that lie strictly inside $[a,b]$, but those that intersect with $[a,b]$, as we will do further in similar cases.
    
\end{remark}
Let us illustrate Lemma \ref{lemmatransit1} with the help of  uniform and Gaussian   distributions. In both cases we consider uniform partitions.
\begin{example}\label{exam4} Let $p(x)=(b-a)^{-1}1_{x\in[a,b]},$ and $t_k^N=a+\frac{(b-a)k}{N}, \,0\le k\le N$. Then $$\entropy_{SH}^N=\sum_{k=1}^N\frac1N\log N=\log N,$$
therefore, entropy increases with a logarithmic rate.
    
\end{example}

\begin{example}\label{exam5}
Let $p(x) = \frac{1}{\sigma\sqrt{2\pi}} \exp\left(- \frac{(x - m)^2}{2\sigma^2}\right)$
denote the density of the Gaussian distribution $\mathcal{N}(m, \sigma^2)$, where $m\in\real, \sigma>0, x\in \real.$  
Consider the partition
\[
\pi_N = \left\{ -N, -N + \frac{1}{N}, -N + \frac{2}{N}, \dots, N - \frac{1}{N}, N \right\}.
\]
Then
\begin{equation}\label{eq:entr-disc-gaus}
\entropy_{SH}^N = - \sum_{k = -N^2}^{N^2 - 1} \Delta F_{k}^N \log \Delta F_{k}^N
- R(N),
\end{equation}
where
\begin{gather*}
\Delta F_{k}^N = \int_{\frac{k}{N}}^{\frac{k+1}{N}} p(x) \, dx, \quad k = -N^2, -N^2 + 1, \dots, N^2 - 1, \end{gather*}
and \begin{equation}\label{residual}R(N):= F(-N) \log F(-N)  
  + \bigl(1 - F(N)\bigr) \log\bigl(1 - F(N)\bigr).
\end{equation}
Obviously,  both $F(-N)$ and $1-F(N)$ tend to zero as $N \to \infty$, and we obtain that
\[
R(N) \to 0, \quad N \to \infty.
\]

By the mean value theorem, for each $k$, there exists $\theta_k^N \in \left(\frac{k}{N}, \frac{k+1}{N}\right)$ such that
\[
\Delta F_{k}^N = p\left(\theta_k^N\right) \cdot \frac{1}{N},
\]
and therefore,
\[
\log \Delta F_{k}^N = \log p\left(\theta_k^N\right) - \log N.
\]
Substituting this identities into \eqref{eq:entr-disc-gaus} yields
\begin{equation}
\entropy_{SH}^N 
= -\sum_{k = -N^2}^{N^2 - 1} \Delta F_{k}^N \log p\left(\theta_k^N\right) 
+ \log N \sum_{k = -N^2}^{N^2 - 1} \Delta F_{k}^N + o(1), \quad N \to \infty.
\label{eq:entr-disc-gaus2}
\end{equation}
Observe that the second sum in \eqref{eq:entr-disc-gaus2} tends to one:
\begin{equation}\label{eq:riemsum1}
\sum_{k = -N^2}^{N^2 - 1} \Delta F_{k}^N 
= \int_{-N}^{N} p(x) \, dx 
\to \int_{-\infty}^{\infty} p(x) \, dx = 1, \quad \text{as } N \to \infty.
\end{equation}
Let us estimate the first sum in \eqref{eq:entr-disc-gaus2}.
For $x \in (\frac{k}{N}, \frac{k+1}{N})$ we have
\begin{align*}
\abs{\log p\left(\theta_k^N\right)}
&= \abs{-\log\left(\sigma\sqrt{2\pi}\right) - \frac{\left(\theta_k^N - m\right)^2}{2\sigma^2}}
\le \abs{\log\left(\sigma\sqrt{2\pi}\right)} + \frac{\left(\theta_k^N - x\right)^2 + \left(x - m\right)^2}{\sigma^2}
\\
&\le \abs{\log\left(\sigma\sqrt{2\pi}\right)} + \frac{1}{\sigma^2 N^2}
+\frac{1}{\sigma^2}\left(x - m\right)^2
\le C +\frac{1}{\sigma^2}\left(x - m\right)^2,
\end{align*}
where $C = \abs{\log\left(\sigma\sqrt{2\pi}\right)} + \sigma^{-2}$.
Using this bound, we estimate the first sum: 
\begin{align}
\MoveEqLeft
\sum_{k = -N^2}^{N^2 - 1} \Delta F_{k}^N \abs{\log p\left(\theta_k^N\right)} 
= \sum_{k = -N^2}^{N^2 - 1} \int_{\frac{k}{N}}^{\frac{k+1}{N}} p(x)  \abs{\log p\left(\theta_k^N\right)} \, dx
\notag\\
&\le C \int_{-N}^{N} p(x)\,dx + \frac{1}{\sigma^2} \int_{-N}^{N} (x-m)^2 p(x)\,dx 
\notag\\
&\le C \int_{-\infty}^{\infty} p(x)\,dx + \frac{1}{\sigma^2} \int_{-\infty}^{\infty} (x-m)^2 p(x)\,dx 
= C + 1.
\label{eq:riemsum2}
\end{align}

Combining \eqref{eq:entr-disc-gaus2}--\eqref{eq:riemsum2}, we conclude that for the Gaussian distribution
\[
\entropy_{SH}^N \sim \log N, \quad N \to \infty,
\]
i.e., the discretized Shannon entropy grows logarithmically with $N$, as in the case of the uniform distribution.
\end{example}

In Examples \ref{exam4} and \ref{exam5} we have chosen a ``moderate'' length of the diameter of partition. Now let us show that, decreasing the interval, we increase the rate of divergence of entropy $\entropy_{SH}^N$ to infinity.
\begin{example}
    
Assume that the density $p(x)$ is bounded and nonzero on the whole $\real$: $p(x)\le C$, $x\in\real$. Since  for $R(N)$ from \eqref{residual} it holds that $R(N)\to 0$   as $N \to \infty,$ we can choose such $N$ that $|R(N)|<1/2$, and additionally $F(N)-F(-N)>1/2.$  Further, consider any positive increasing unbounded   sequence $A_N$ such that $e^{A_N}\in \mathbb{N}$ and $ A_N-\log N\to \infty$ when $N\to\infty$, and choose a partition of the form $x_k^N=-N+\frac{2kN}{e^{A_N}},\,0\le k\le e^{A_N}$. Then $\Delta F_{k}^N\le \frac{2CN}{e^{A_N}},$ whence $$\entropy_{SH}^N\ge \frac12\left(A_N-\log(2C)-\log N\right)\sim \frac12 A_N,$$ so, we indeed can achieve any rate of divergence.

\end{example}

So, we see that the formulas for Shannon entropy for discrete and continuous distributions are, in some sense,  incompatible.   As it was mentioned in Section \ref{sec-intr}, one of the attempts to adjust the notions of discrete and continuous Shannon entropies was made by E.\,T.~Jaynes \cite{Jaynes62} by introducing limiting density of discrete points.
This notion has the following form: let we have a set of $N$ discrete points such that
\[
\lim_{N\to\infty} \frac1N \bigl( \text{number of points in } (a, b) \bigr)
= \int_a^b m(x)\,dx,
\]
where $m$ is some non-negative integrable function. Then the respective entropy is defined as the value having the following  asymptotic behavior:   
\[
\entropy_N \sim \log N - \int_\real p(x) \log \frac{p(x)}{m(x)} \,dx,
\quad N \to \infty.
\]
Having a term $\log N$, $\entropy_N$ in inconvenient to use in rigorous mathematical calculations.
In this connection we propose a bit another approach to the definition of differential Shannon  entropy.

\section{  Alternative versions of Shannon  entropy and their properties. Discrete functionals compatible with differential Shannon entropy and with its alternatives}\label{sect.4}
From now on, we consider the distributions with density satisfying the assumption 
\[
 \int_\real p(x) \abs{\log p(x)}\,dx < \infty.
\]
Having established that even the discretization of a continuous distribution leads to entropies that grow to infinity, we abandon the attempt to relate discrete and continuous entropies but instead we consider three alternatives to continuous Shannon entropy. All of them are strictly positive, do not contain any unbounded terms and, what is even more important, have the behavior with respect to the parameters of distribution that are similar to Shannon entropy.

Consider the following alternative functionals to standard Shannon entropy of absolutely continuous distribution. They are created by analogy  with original formula \eqref{entrShan0}. More precisely, let  
\begin{gather}\label{alterShann}
\entropy_{SH}^{(1)}(\{p(x), x \in \real\}) = \int_\real p(x) \abs{\log p(x)} \, dx;
\\
\label{alterShann2}
\entropy_{SH}^{(2)}(\{p(x), x \in \real\}) = \int_\real p(x)  (-\log p(x))_+  \, dx;
\\
\label{alterShann3}
\entropy_{SH}^{(3)}(\{p(x), x \in \real\}) = \int_\real p(x) \log \bigl((p(x))^{-1} + 1\bigr) \, dx;
\end{gather}
and, if $p(x)$ is bounded,
\[
\entropy_{SH}^{(4)}(\{p(x), x \in \real\}) = \int_\real \frac{p(x)}{M} \log \frac{M}{p(x)} \, dx,
\quad \text{where }M = \sup_{x \in \real} p(x).
\]

Obviously, $\entropy_{SH}^{(i)}$ are strictly positive for $i = 1, 2,3$,
$\entropy_{SH}^{(4)} \ge 0$ and equals zero only if
$p(x) = \frac{1}{b-a}$, $x \in [a, b]$. (In any case we integrate over $\supp p(x)$.)
Now, let us establish how alternative functionals can be obtained as the limits of the discretization procedure. 
\begin{theorem}\label{theor1}
    Let the function $p = p(x)$, $x \in \real$ be a density of probability distribution, satisfying the assumptions
\begin{itemize}
\item[$(B)$ $(i)$] $p\in C(\real)$   and  $\int_\real p(x)|\log p(x)|dx<\infty$;

 \item[$(ii)$] there exist $ x_0>0$ and $K >0$ such   that for any $D>x_0$ and  $|x|,|y|\in( x_0, D) $    it holds that   $ p(x)\le K p(y)$ if $|x-y|\le 1/D$.
 \end{itemize} 
  Then, for the sequence of the uniform partitions  satisfying  Assumption (A), 
\begin{equation}\label{firstlimits}
 \entropy_{SH} = \lim_{N\to\infty} \sum_{k= 1}^{k_N} \Delta F_k^N  {\log \left(\frac{\Delta F_k^N}{\Delta x_k^N}\right)},\,\,\text{and}\,\,\,   
\entropy_{SH}^{(1)} = \lim_{N\to\infty} \sum_{k= 1}^{k_N} \Delta F_k^N \abs{\log \left(\frac{\Delta F_k^N}{\Delta x_k^N}\right)},
\end{equation} 
where we put $\log \left(\frac{\Delta F_k^N}{\Delta x_k^N} \right)=0$ if $\Delta F_k^N=0$.
\end{theorem}
\begin{remark} We wrote the relation \eqref{firstlimits} in its initial form, however, since the partitions are uniform, it can be simplified to 
\begin{equation*} 
 \entropy_{SH} = \lim_{N\to\infty} \sum_{k= 1}^{k_N} \Delta F_k^N  {\log \left(N {\Delta F_k^N} \right)},\,\,\text{and}\,\,\,   
\entropy_{SH}^{(1)} = \lim_{N\to\infty} \sum_{k= 1}^{k_N} \Delta F_k^N \abs{\log \left( {N\Delta F_k^N} \right)},
\end{equation*}
    
\end{remark}
\begin{proof} Both equalities in \eqref{firstlimits} are proved similarly. Since we are interested in alternative entropies, we shall prove the 2nd equality. 
For technical simplicity assume that $ p(x)>0$, $x\in \real$. 
Choose $\varepsilon > 0$ and $x_1 > 0$ such that
\[
\int_{\abs{x} \ge x_1} p(x) \,dx
+ \int_{\abs{x} \ge x_1} p(x) \abs{\log p(x)}\,dx < \varepsilon.
\]
Put $N > x_0 \vee x_1$ (then $\int_{\abs{x} \ge N} p(x) \,dx
+ \int_{\abs{x} \ge N} p(x) \abs{\log p(x)}\,dx < \varepsilon$), and consider the difference
\[
\Delta_N = \abs{\entropy_{SH}^1 - \sum_{k=1}^{k_N} \Delta F_k^N \left|\log \left(\frac{\Delta F_k^N}{\Delta x_k^N}\right)\right|}
= \abs{\entropy_{SH}^1 - \sum_{k=1}^{k_N} \Delta F_k^N \left|\log p\left(\theta_k^N\right)\right|},
\]
where $\theta_k^N \in [x_{k-1}^N, x_k^N]$.
Also, recall that $x_k^N-x_{k-1}^N=\frac1N$.
Denote $z=x_1 \vee x_0$. It is possible to bound $\Delta_N$ as follows: 
\begin{align*}
\Delta_N &\le \varepsilon+\abs{\int_{\abs{x} \le z} p(x) |\log p(x)|\, dx
- \sum_{k:\left[x_{k-1}^N, x_k^N\right] \cap [-z, z] \ne \emptyset} \Delta F_k^N \left|\log p\left(\theta_k^N\right)\right|}
\\
&\quad + \abs{\int_{z < \abs{x} < N } p(x) |\log p(x)|\,dx
- \sum_{k:\left[x_{k-1}^N, x_k^N\right] \cap [-z, z] = \emptyset} \Delta F_k^N \left|\log p\left(\theta_k^N\right)\right|}=\varepsilon+I_1^N+I_2^N,
\end{align*}
where $I_1^N$ is, in some sense, the main term, and $I_2^N$ is a reminder term. 
We start with  $I_1^N$. Obviously, 
\[
\sum_{k:\left[x_{k-1}^N, x_k^N\right] \cap [-z, z] \ne \emptyset} \Delta F_k^N \left|\log p\left(\theta_k^N\right)\right|
= \int_{x_{k_{z}^1-1}^N}^{x_{k_{z}^2}^N}
p(x) \left|\log p\left(\theta_{x}^N\right)\right| dx,
\]
where $x_{k_{z}^1-1}^N$ is the left endpoint of the first interval $[x_{k-1}^N, x_k^N]$ such that
$[x_{k-1}^N, x_k^N] \cap [-z, z] \ne \emptyset$,
  $x_{k_{z}^2}^N$ is the right endpoint of the last of such intervals if to consider them from the left to the right, and
$\theta_{x}^N = \theta_k^N$ if $x \in [x_{k-1}^N, x_k^N]$.
Note that for any $N>1$   we have that $x_{k_{z}^1-1}^N>z-1$ and $x_{k_{z}^2}^N<z+1.$
Then   continuity of $p(x)$ and condition $(i)$ supply the  existence of  $\delta_2>\delta_1>0$ such that for any $N>1$ 
and for all    $x\in [x_{k_{z}^1-1}^N, x_{k_{z}^2}^N]$ it holds that 
$\delta_2\ge p\left(\theta_{x}^N\right)\ge  \delta_1$,  and consequently, $$p(x) \left|\log p\left(\theta_{x}^N\right)\right|\le p(x) (|\log(\delta_1)|\vee|\log(\delta_2)|).$$
 
Also,  $x_{k_{z}^1-1}^N\uparrow -z$ and $x_{k_{z}^2}^N\downarrow z$ as $N\to\infty$ because $|x_{k_{z}^1-1}^N+z|\le |\pi_N|$ and $|x_{k_{z}^2}^N-z|\le |\pi_N|$.  In turn, it means that
\begin{equation}\label{convint}
 \int_{x_{k_{z}^1-1}^N}^{x_{k_{z}^2}^N}
p(x) \log p\left(\theta_{x}^N\right) dx 
\to \int_{-z}^{z}
p(x) \log p(x) dx,   
\end{equation}

as $N\to\infty$, where we applied the convergence
\begin{gather*}
\left[x_{k_{z}^1-1}^N, x_{k_{z}^2}^N\right]
\to [-z, z],\,\,
 p(x) |\log p\left(\theta_{x}^N\right)| 
\to p(x) |\log p(x)|
\end{gather*}
and the Lebesgue dominated convergence theorem.

Now, consider the remainder term $I_2^N$. It can be divided into two parts and bounded as follows:
\[
I_2^N \le I_{21}^N + I_{22}^N,
\]
where
\begin{align*}
I_{21}^N & = \abs{\int_{z < x < N } p(x) |\log p(x)|\,dx
- \sum_{k:\left[x_{k-1}^N, x_k^N\right] \subset [z, N]} \Delta F_k^N \left|\log p\left(\theta_k^N\right)\right|},
\\
I_{22}^N & = \abs{\int_{-N < x <- z} p(x) |\log p(x)|\,dx
- \sum_{k:\left[x_{k-1}^N, x_k^N\right] \subset [-N, -z]} \Delta F_k^N \left|\log p\left(\theta_k^N\right)\right|}.
\end{align*}

Let us bound $I_{21}^N$, and $I_{22}^N$ can be considered similarly.
Let $x_{(z)}^N$ denote the left endpoint of the first interval $[x_{k-1}^N, x_k^N]$ such that
$[x_{k-1}^N, x_k^N] \subset [z, N]$,
and let $x_{(N)}^N$ denote the right endpoint of the last of such intervals, if to consider them from the left to the right.
As before, $\theta_x^N = \theta_k^N$ if $x \in [x_{k-1}^N, x_k^N]$.
Then
\[
I_{21}^N \le 2\varepsilon + \abs{\int_{x_{(z)}^N}^{x_{(N)}^N} p(x) |\log p(x)|\,dx
- \int_{x_{(z)}^N}^{x_{(N)}^N} p(x) \abs{\log p\left(\theta_x^N\right)}\,dx}
= 2 \varepsilon + I_3^N,
\]
where we take into account that
\begin{align*}
\MoveEqLeft
\abs{\int_{z < x < N} p(x) |\log p(x)|\,dx
- \int_{x_{(z)}^N}^{x_{(N)}^N} p(x) |\log p(x)|\,dx}
\\
&\le \int_{z < x < x_{(z)}^N} p(x) |\log p(x)|\,dx
+ \int_{x_{(N)}^N}^N p(x) |\log p(x)|\,dx
< 2 \varepsilon.
\end{align*}
Furthermore,
\[
I_3^N \le \int_{x_{(z)}^N}^{x_{(N)}^N} p(x) \abs{\log \frac{p(x)}{p\left(\theta_x^N\right)}}\,dx.
\]
Now, according to condition $(ii)$,
\[
p(x) \le K p\left(\theta_x^N\right)
\quad\text{and}\quad
p\left(\theta_x^N\right) \le K p(x),
\]
because
$N > x_0$, $|\theta_x^N - x| < \frac1N$, $\theta_x^N \le N$, $x \le N$. 
Therefore,
\[
I_3^N \le  |\log K|  \int_{x_{(z)}^N}^{x_{(N)}^N} p(x)\,dx
\le |\log K| \varepsilon,
\]
because $x_{(z)}^N > x_1$.

Since $\varepsilon > 0$ is arbitrary, the proof follows.
\end{proof}
\begin{remark} Condition $(ii)$ is not as sophisticated as it seems. Of course, it holds for the densities with compact support. It also holds, for example, for the Gaussian  distribution. Indeed, let
\[
p(x) = \frac{1}{\sigma\sqrt{2\pi}} \exp\set{- \frac{(x-m)^2}{2\sigma^2}},
\quad m\in\real, \; \sigma > 0, \; x\in \real.
\]
Consider the inequality
\[
p(x) \le K p(y)
\]
or, that is the same,
\begin{equation}\label{eq:ineq1}
\exp\set{- \frac{(x-m)^2}{2\sigma^2}}
\le K \exp\set{- \frac{(y-m)^2}{2\sigma^2}}.
\end{equation}
Inequality \eqref{eq:ineq1} is equivalent to the following one:
\begin{equation}\label{eq:ineq2}
(y - m)^2 - (x - m)^2 \le 2\sigma^2 \log K.
\end{equation}
Of course, \eqref{eq:ineq2} will hold if
\begin{equation}\label{eq:ineq3}
\abs{y - x} \bigl(\abs{x} + \abs{y} + 2\abs{m}\bigr)
\le 2\sigma^2 \log K.
\end{equation}
Taking into account that we should consider
$\abs{y - x} \le \frac{1}{D}$ and $\abs{x} \le D$, $\abs{y} \le D$,
we see that \eqref{eq:ineq3} will be satisfied if
\[
\frac{1}{D} (2D + 2\abs{m}) \le 2\sigma^2 \log K,
\]
or
\[
1 + \frac{\abs{m}}{D} \le \sigma^2 \log K.
\]
If we put $x_0 =|m|$  and $K = \exp\{\frac{2}{\sigma^2}\}$, then assumption $(ii)$ is fulfilled. Moreover, in fact, inequality $p(x) \le K p(y)$ will be fulfilled for all 
$|x|,|y|\in(0, D) $      if $|x-y|<1/D$ and  $D>|m|$.

The case of exponential distribution with mean $\mu > 0$ is even simpler. In this case inequality $p(x) \le K p(y)$, or, that is the same, $\mu^{-1} e^{- x/\mu}\le K\mu^{-1} e^{- y/\mu}$ is fulfilled for $x>y$ with $K=1$ and for $0<y-x<D$ if $D>x_0=1/\mu$ and $K=e$.

Another example: assume that there exists $x_1 > 0$ such that $p(x)$ is increasing on $(-\infty, -x_1)$, decreasing on $(x_1, \infty)$, and for any $t \in \real$
\[
\lim_{\abs{x} \to \infty} \frac{p(x+t)}{p(x)} = 1.
\]

Consider, for example, $x > x_1$ and $y > x_1$ and write the inequality
\[
p(x) \le K p(y).
\]
Of course, it holds for $x > y$ with $K = 1$.
Therefore, let $x < y$.
Choose $t = 1$ and $x_2 > 0$ such that
\[
\frac{p(x+1)}{p(x)} > \frac12
\quad\text{for all } x > x_2.
\]
Now, choose
$D > x_1 \vee x_2 \vee 1$.
Then for $y - x < \frac1D$
\[
\frac{p(y)}{p(x)} > \frac{p\left(x + \frac1D\right)}{p(x)}
> \frac{p(x + 1)}{p(x)} > \frac12,
\]
whence $p(x) < 2p(y)$. The example of such density: $p(x)=C_1(1+x^2)^{-1}$.

Assumption $(ii)$ is not fulfilled, for example, for
\[
p(x) = C_2 e^{-x^4}, \quad x \in \real,
\]
because in this case inequality
\[
p(x) \le K p(y)
\]
is equivalent to
\[
y^4 - x^4 \le \log K,
\]
or
\begin{equation}\label{equ.11}
\abs{y-x} \abs{y+x} \left(x^2 + y^2\right) \le \log K.
\end{equation}
If $\abs{y - x} \le \frac1D$ and $\abs{y} \le D$, $\abs{x} \le D$, we still have in the left-hand side of \eqref{equ.11} the value $x^2 + y^2$ that can increase to $+\infty$.
It does not mean that it is impossible to construct a prelimit sum of the form $\sum_{k= 1}^{k_N} \Delta F_k^N \abs{\log \left(\frac{\Delta F_k^N}{\Delta x_k^N}\right)}$ that will converge to $\entropy_{SH}^{(1)}$, but it is necessary to consider  partition of diameter $N^{-3}$ instead of $N^{-1}$.
\end{remark}
The next result can be  proved by the same steps as Theorem \ref{theor1}, therefore we omit the proof. 
\begin{theorem} Let the function $p = p(x)$, $x \in \real$ be a density of probability distribution, satisfying the assumptions $(B)$.
 Then, for the sequence of uniform partitions  satisfying  Assumption (A), 
\[
\entropy_{SH}^{(2)} = \lim_{N\to\infty} \sum_{k= 1}^{k_N} \Delta F_k^N \left(\log \left(\frac{\Delta F_k^N}{\Delta x_k^N}\right)\right)_+,
\]
\[
\entropy_{SH}^{(3)} = \lim_{N\to\infty} \sum_{k= 1}^{k_N} \Delta F_k^N \left(\log \left(\frac{\Delta F_k^N}{\Delta x_k^N}\right)+1\right),
\]
where we put $\log \left(\frac{\Delta F_k^N}{\Delta x_k^N} \right)=0$ if $\Delta F_k^N=0$.
\end{theorem}
\section{The behavior of alternative versions of differential Shannon entropy as the  functions of parameters of distributions}\label{sec5.behav}    
As we already claimed, a significant advantage of alternative entropies is their strict positivity. Now let us analyze their behavior as functions of the parameters of some distributions and compare them with the corresponding behavior of the standard Shannon entropy.
\subsection{Gaussian distribution} 
Consider Gaussian  distribution with zero mean, for technical simplicity.
So, let 
\begin{equation}\label{eq:norm-pdf}
p_0(x) = \frac{e^{-\frac{x^2}{2 \sigma^2}}}{\sigma \sqrt{2\pi}},
\quad x \in \real, \; \sigma > 0.
\end{equation}
Recall that
\[
\entropy_{SH} \bigl(\{p_0(x), x \in \real\}\bigr) = \frac12(1+  \log 2\pi)+\log \sigma,
\]
and therefore, as the function of $\sigma$, it increases from $-\infty$ to $+\infty$ as $\sigma$ increases from 0 to $\infty$. Monotonicity is a convenient property, while, as we said, negative entropy or zero entropy of the non-degenerate distribution is not a logical phenomenon. Now let us consider the behavior of
$\entropy_{SH}^{(i)}(\{p_0(x), x \in \real\})$
as functions of $\sigma$ and clarify  advantages and disadvantages of these alternative entropies.

\begin{proposition}
\leavevmode
\begin{enumerate}[1)]
\item 
$\entropy_{SH}^{(1)}(\{p_0(x), x \in \real\})$ decreases in $\sigma \in (0, \sigma_0)$ and increases in $\sigma\in(\sigma_0,+\infty)$, where $\sigma_0 \approx 0.317777$ is the unique value for which
\[
\int_0^{-\log\left(\sigma_0\sqrt{2\pi}\right)} \frac{e^{-z}}{\sqrt{z}}\,dz
= \int_{-\log\left(\sigma_0\sqrt{2\pi}\right)}^{\infty} \frac{e^{-z}}{\sqrt{z}}\,dz.
\]
(Obviously,  $-\log\left(\sigma_0\sqrt{2\pi}\right)>0$.)
\item 
$\entropy_{SH}^{(i)}(\{p_0(x), x \in \real\}),\,i=2,3,4$ strictly increase in $\sigma>0$ from 0 to $+\infty$.

\end{enumerate}
\end{proposition}

\begin{proof}
1) Note that
\begin{align*}
\entropy_{SH}^{(1)}(\{p_0(x), x \in \real\})
&= \int_\real \frac{e^{-\frac{x^2}{2 \sigma^2}}}{\sigma \sqrt{2\pi}}
\abs{-\frac{x^2}{2 \sigma^2} - \log\left(\sigma \sqrt{2\pi}\right)} dx
\\
&= 2 \int_0^\infty \frac{e^{-\frac{x^2}{2 \sigma^2}}}{\sigma \sqrt{2\pi}}
\abs{\frac{x^2}{2 \sigma^2} + \log\left(\sigma \sqrt{2\pi}\right)} dx
\\
&= \biggl| \frac{x}{\sigma \sqrt2} = y \biggr|
= \frac{2}{\sqrt\pi} \int_0^\infty e^{-y^2}
\abs{y^2 + \log\left(\sigma \sqrt{2\pi}\right)} dy
\\
&= \biggl| y^2 = z \biggr|
= \frac{1}{\sqrt\pi} \int_0^\infty \frac{e^{-z}}{\sqrt{z}}
\abs{z + u} dz
\eqqcolon f(u),
\end{align*}
where $u = \log(\sigma \sqrt{2\pi})$. Now it is sufficient to investigate monotonicity of $f$ in $u$.

If $u > 0$, i.e., $\sigma > \frac{1}{\sqrt{2\pi}}$,
$f(u) = \frac{1}{\sqrt\pi} \int_0^\infty \frac{e^{-z}}{\sqrt{z}}
(z + u) dz$
strictly increases in $u$ from
$f(0) = \frac{1}{\sqrt\pi} \int_0^\infty e^{-z} \sqrt{z}\,
dz = \frac{\Gamma(3/2)}{\sqrt\pi} = \frac12$
to $+\infty$.

Now, let $u \le 0$, i.e., $\sigma \le \frac{1}{\sqrt{2\pi}}$.
Then
\[
f(u) = \frac{1}{\sqrt\pi} \int_0^{-u} \frac{e^{-z}}{\sqrt{z}} (-z - u) \,dz + \frac{1}{\sqrt\pi} \int_{-u}^\infty \frac{e^{-z}}{\sqrt{z}} (z + u)\, dz.
\]
Therefore
\[
f'(u) = - \frac{1}{\sqrt\pi} \int_0^{-u} \frac{e^{-z}}{\sqrt{z}}\,dz + \frac{1}{\sqrt\pi} \int_{-u}^\infty \frac{e^{-z}}{\sqrt{z}} \, dz,
\]
and
\[
f''(u) = \frac{2}{\sqrt\pi}\,\frac{e^u}{\sqrt{-u}} > 0,\,u<0.
\]
It means that $f'(u)$ strictly increases in $u \in (-\infty, 0)$ from
$f'(-\infty) = - \frac{1}{\sqrt\pi} \int_{0}^\infty \frac{e^{-z}}{\sqrt{z}} \, dz = - \frac{\Gamma(1/2)}{\sqrt\pi} = -1$
to
$f'(0) = \frac{1}{\sqrt\pi} \int_{0}^\infty \frac{e^{-z}}{\sqrt{z}} \, dz = 1$ having only one zero value inside, and it will be the point of minimum.

In turn, it means that $f(u)$ decreases from $+\infty$ to $f(u_0)$, where $u_0$ is defined as the unique value for which
\[
\int_0^{-u_0} \frac{e^{-z}}{\sqrt{z}}\,dz = \int_{-u_0}^\infty \frac{e^{-z}}{\sqrt{z}} \, dz.
\]
Solving this equation numerically, we get $u_0 \approx -0.227468$, which corresponds to $\sigma_0 = e^{u_0}/\sqrt{2\pi} \approx 0.317777$.

 Finally, it means that entropy $\entropy_{SH}^{(1)}(\{p_0(x), x \in \real\})$ decreases from $+\infty$ to $0.428674$ when $\sigma$ increases from $0$ to $\sigma_0$ and increases from $0.428674$ to $+\infty$ when $\sigma$ increases from   $\sigma_0$ to $+\infty.$
 \begin{figure}
     \centering
     \includegraphics[width=0.7\linewidth]{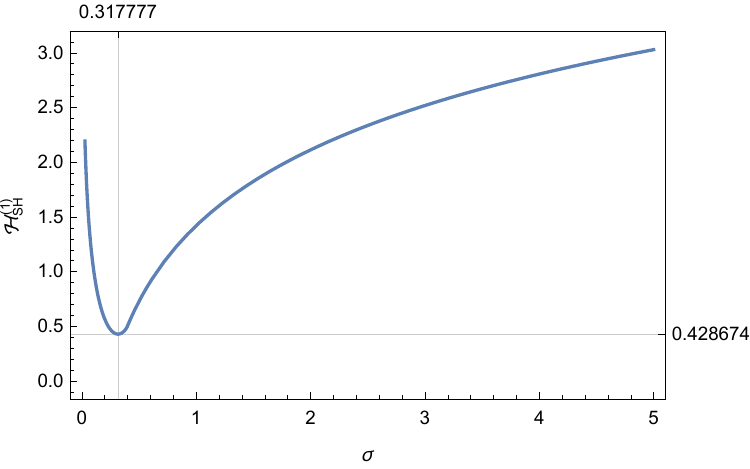}
     \caption{$\entropy_{SH}^{(1)}$ for Gaussian distribution as a function of $\sigma$}
     \label{fig:H1-gauss}
 \end{figure}
 
2) These cases are similar to each other and simpler than 1).
Indeed, for $u = \log (\sigma \sqrt{2\pi})$
\[
\entropy_{SH}^{(2)} \bigl(\{p_0(x), x \in \real\}\bigr)
= \int_\real \frac{e^{-\frac{x^2}{2\sigma^2}}}{\sigma\sqrt{2\pi}}
\left(\frac{x^2}{2\sigma^2} + u\right)_+ dx
= \frac{1}{\sqrt\pi}\int_0^\infty \frac{e^{-z}}{\sqrt{z}}
\left(z + u\right)_+ dz.
\]
Obviously, the value $(z + u)_+$ and consequently, the integral strictly increase in $u$ (and so in $\sigma > 0$), and integral increases from 0 to $+\infty$.

Similarly,
\[
\entropy_{SH}^{(3)} \bigl (\{p_0(x), x \in \real\}\bigr)
= \int_\real \frac{e^{-\frac{x^2}{2\sigma^2}}}{\sigma\sqrt{2\pi}} \log\left(\sigma\sqrt{2\pi} e^{\frac{x^2}{2\sigma^2}} + 1\right) dx
= \frac{1}{\sqrt\pi}\int_0^\infty \frac{e^{-z}}{\sqrt{z}} \log\left(\sigma\sqrt{2\pi} e^z + 1\right) dz,
\]
and this function also strictly increases in $\sigma > 0$, from 0 to $+\infty$.

Entropy $\entropy_{SH}^{(4)} (\{p_0(x), x \in \real\})$ was calculated in \cite{BMR}, it equals $\sigma \sqrt{\frac{\pi}{2}}$.
\end{proof}

\begin{figure}[h]
     \centering
     \includegraphics[width=0.7\linewidth]{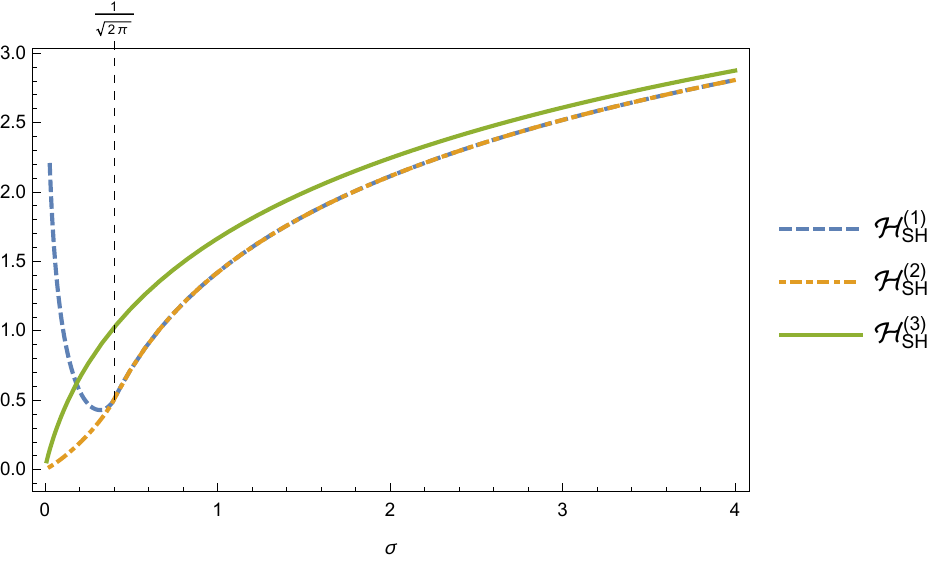}
     \caption{$\entropy_{SH}^{(i)}$, $i = 1,2,3$, for Gaussian distribution as functions of $\sigma$}
     \label{fig:H123-gauss}
 \end{figure}
 
\begin{remark}
1) It is clear that
$\entropy_{SH}^{(1)} (\{p_0(x), x \in \real\}) = \entropy_{SH}^{(2)} (\{p_0(x), x \in \real\})$ for $\sigma > \frac{1}{\sqrt{2\pi}}$.

2) Advantages of all entropies $\entropy_{SH}^{(i)} (\{p_0(x), x \in \real\})$ are their positive values.

Disadvantage of $\entropy_{SH}^{(1)} (\{p_0(x), x \in \real\})$ is the fact that it admits the same value for two different variances.
Therefore, having the value of this entropy we should have some additional information about   $\sigma$ in order to distinguish these two values.

Advantages of $\entropy_{SH}^{(i)} (\{p_0(x), x \in \real\})$, $i =2,3,4$, is their strict increasing in $\sigma > 0$.
\end{remark}

\subsection{Exponential distribution} Consider exponential distribution with the density $p_1(x) = \mu^{-1} e^{- x/\mu}$, $x\ge 0$, $\mu > 0$, in the same spirit as Gaussian distribution. Recall that its standard Shannon entropy equals
\[
\entropy_{SH} (\{p_1(x), x\ge 0\}) = -\int_0^{\infty} \mu^{-1} e^{-x/\mu} \log \left(\mu^{-1} e^{-x/\mu}\right) dx
= \int_0^\infty e^{-y}(\log \mu + y)\,dy = 1 + \log \mu,
\]
and increases from $-\infty$ to $+\infty$ when $\mu$ increases from $0$ to $+\infty$. 
Since the next statements are the results of the straightforward calculations, we omit the proofs.
\begin{proposition} 
Let $p_1(x) = \mu^{-1} e^{-x/\mu}$, $x\ge 0$, $\mu > 0$.
\begin{enumerate}[1)]
\item $\displaystyle \entropy_{SH}^{(1)} \bigl(\{p_1(x), x\ge 0\}\bigr)
= \int_0^\infty \mu^{-1} e^{-x/\mu} \abs{\log \left(\mu^{-1} e^{- x/\mu}\right)} dx
= \begin{cases}
2\mu -\log\mu - 1, & \mu \le 1, \\
1 + \log\mu, & \mu > 1, 
\end{cases}$
it decreases from $+\infty$ to $\log 2$, when $\mu$ increases from 0 to $1/2$ and increases from $\log 2$ to $+\infty$ when $\mu$ increases from $1/2$ to $+\infty$.

\item $\displaystyle \entropy_{SH}^{(2)} \bigl(\{p_1(x), x \ge 0\}\bigr)
= \int_0^\infty \mu^{-1} e^{-x/\mu}\left(-\log \left(\mu^{-1} e^{-x/\mu}\right)\right)_+ dx
= \begin{cases}
\mu, & \mu \le 1,\\
1 + \log \mu, & \mu > 1,
\end{cases}$
and it increases from 0 to $+\infty$ when $\mu$ increases from 0 to $+\infty$.

\item $\displaystyle \entropy_{SH}^{(3)} \bigl(\{p_1(x), x\ge 0\}\bigr)
= \int_0^\infty \mu^{-1} e^{- x/\mu}\log \left(\mu e^{x/\mu} + 1\right) dx
= \log (\mu + 1) + \mu \log \left(\frac1\mu+1\right)$
and increases from 0 to $+\infty$ when $\mu$ increases from 0 to $+\infty$

\item $\displaystyle \entropy_{SH}^{(4)} \bigl(\{p_1(x), x\ge 0\}\bigr)
= \int_0^\infty e^{-x/\mu} \log e^{x/\mu} \, dx
= \mu$
and increases from 0 to $+\infty$ when $\mu$ increases from 0 to $+\infty$.
\end{enumerate}
\end{proposition}

\begin{figure}[h]
     \centering
     \includegraphics[width=0.7\linewidth]{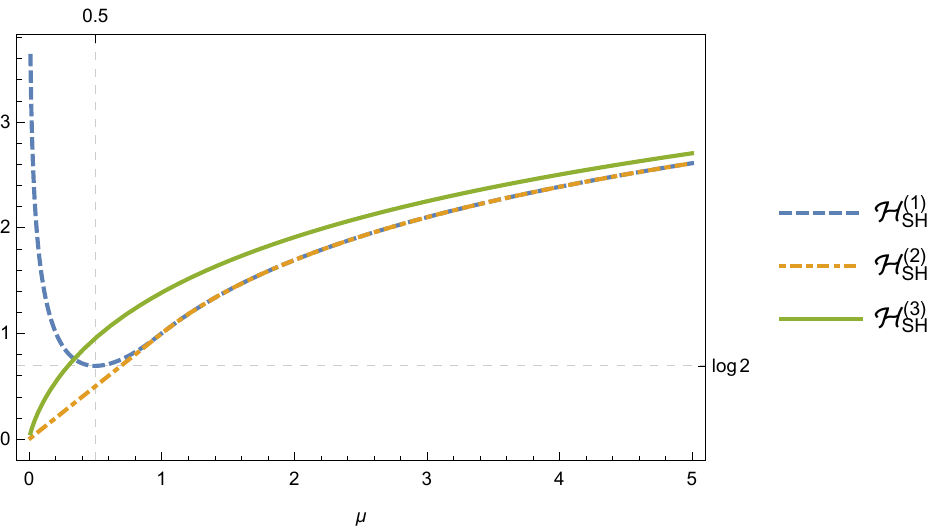}
     \caption{$\entropy_{SH}^{(i)}$, $i = 1,2,3$, for exponential distribution as functions of $\mu$}
     \label{fig:H123-exp}
 \end{figure}

\FloatBarrier

\section{R\'{e}nyi entropy: choice of the pre-limit functionals and  alternative forms}\label{sec6}
Consider in the same context as Shannon entropy , but more briefly, standard and alternative R\'{e}nyi entropies.

\begin{definition}
R\'enyi entropy with index $\alpha > 0$, $\alpha \ne 1$, for discrete distribution equals 
\[
\entropy_{R,\alpha} (\set{p_k, k \ge 1}) = \frac{1}{1-\alpha} \log \left(\sum_{k \ge 1} p_k^\alpha\right),
\]
and for continuous distribution it equals
\[
\entropy_{R,\alpha} (\set{p(x), x \in \real}) = \frac{1}{1-\alpha} \log \left(\int_\real p^\alpha(x) \,dx\right).
\]
\end{definition}

\begin{remark}
While for discrete distribution multiplier $\frac{1}{1 - \alpha}$ is natural, because $\log(\sum_{k \ge 1} p_k^\alpha)$ is positive for $\alpha < 1$ and negative for $\alpha > 1$, for continuous distribution it does not play a role of a factor that corrects the sign. For example, R\'enyi entropy of the normal distribution equals (see, e.g., \cite{entropy-gaussian})
\[
\entropy_{R,\alpha} = \log\sigma + \frac12 \log(2\pi) + \frac{\log\alpha}{2(\alpha - 1)},
\]
and can be both negative and positive, increasing from $-\infty$ to $+\infty$ with $\sigma \in (0,+\infty)$, for any $\alpha > 0$.
Nevertheless, traditionally factor $\frac{1}{1-\alpha}$ is preserved, because
\[
\entropy_{R,\alpha} (\set{p(x), x \in \real}) \to \entropy_{SH} (\set{p(x), x \in \real}),
\]
as $\alpha \to 1$, under some additional assumptions.
\end{remark}
\subsection{Incompatibility of R\'{e}nyi  entropy   with its discrete counterpart}
As the first result, we prove that, in general, R\'enyi entropies for discrete and continuous distributions are incomparable in the same sense as the respective Shannon entropies.

\begin{lemma}\label{lemmatransit2}
\begin{itemize}
    \item[$(i)$] Let $0<\alpha<1$ and $p(x)$, $x \in \real$  be a continuous   density of probability distribution. 
Then, in terms of Assumption (A), 
\begin{gather*}
\entropy_{R,\alpha}^N \coloneqq \frac{1}{1-\alpha} \log \left(\sum_{k=1}^{k_N} \left(\Delta F_k^N\right)^\alpha+\left(F\left(x_0^N\right)\right)^\alpha 
+ \Bigl(1 - F\left(x_{k_N}^N\right)\Bigr)^\alpha\right) \to +\infty,
\quad \text{as } N \to \infty.\end{gather*}
 \item[$(ii)$] Let $ \alpha>1$ and $p(x)$, $x \in \real$  be a bounded   density of probability distribution. 
Then, in terms of Assumption (A), 
\begin{gather*}  \entropy_{R,\alpha}^N \to -\infty,
\quad \text{as } N \to \infty.
\end{gather*}
\end{itemize}
Here we put $ \log 0=0$. 
\end{lemma}
 \begin{remark} Of course, we will obtain the same result for the simplified sum 
   $$\widetilde\entropy_{R,\alpha}^N=  \frac{1}{1-\alpha} \log \left(\sum_{k=1}^{k_N} \left(\Delta F_k^N\right)^\alpha  \right), $$
   therefore, for the technical simplicity, we shall consider $\widetilde\entropy_{R,\alpha}^N$ in what follows.
\end{remark}
\begin{proof}
$(i)$ We follow the proof of Lemma \ref{lemmatransit1} and find the interval $[a,b] \subset \supp\set{p(x), x \in \real}$ such that for $k_1^N < k \le k_2^N$ it holds that
$0 < m_1 \Delta_k^N \le \Delta F_k^N \le M_1 \Delta_k^N \le M_1 \abs{\pi_N}$.
Then for $0 < \alpha < 1$
\[
\sum_{k=1}^{k_N} \left(\Delta F_k^N\right)^\alpha
\ge \sum_{k: t_k \in [a,b]} \frac{\Delta F_k^N}{\left(\Delta F_k^N\right)^{1-\alpha}}
\ge \frac{1}{\left(M_1 \abs{\pi_N}\right)^{1-\alpha}}
\Bigl(F\left(x_{k_2}^N\right) - F\left(x_{k_1}^N\right)\Bigr).
\]
As $N \to \infty$,
\[
\frac{1}{\left(M_1 \abs{\pi_N}\right)^{1-\alpha}} \to +\infty,
\quad\text{and}\quad
F\left(x_{k_2}^N\right) - F\left(x_{k_1}^N\right)
\to F(b) - F(a) > 0,
\]
whence $\widetilde\entropy_{R,\alpha}^N \to +\infty$.

$(ii)$ Now, let $\alpha > 1$.
Then
\[
\sum_{k=1}^{k_N} \left(\Delta F_k^N\right)^\alpha
\le (M\abs{\pi_N})^{\alpha-1} \to 0
\quad \text{as } N \to \infty,
\]
whence the proof follows.
\end{proof}

\subsection{Discrete functionals compatible with  R\'{e}nyi entropy of continuous distribution}

In the framework of Assumption (A) consider the discrete functional
\[
\widetilde\entropy_{R,\alpha}^N = \frac{1}{1-\alpha} \log\left(\sum_{k=1}^{k_N} \left(\Delta F_k^N\right)^\alpha \left(\Delta_k^N\right)^{1-\alpha} \right).
\]
From now on, we assume that
$\int_\real p^\alpha (x)\,dx < \infty$
for $\alpha \in (0, +\infty)$ in consideration. Obviously, this integral is strictly positive.

\begin{theorem}\label{theor3}
\begin{itemize}
\item[$(i)$]
Let $\alpha > 1$, $p \in C(\real)$.
Then
\begin{equation}\label{eq:renyi-conv}
\widetilde\entropy_{R,\alpha}^N \to \entropy_{R,\alpha} (\set{p(x),x \in \real}),
\quad\text{as } N \to \infty.
\end{equation}

\item[$(ii)$]
Let $0 < \alpha < 1$. Consider only uniform partitions $\pi_N=\{\frac kN, k = -N^2, \dots,  N^2\}$ and assume that $p \in C(\real)$ and there exists such $z>0$ that $p$ increases on $(-\infty, -z]$ and decreases on $[z, \infty)$.
Then \eqref{eq:renyi-conv} holds.
\end{itemize}
\end{theorem}
\begin{remark}\label{imprem} As an intermediate result, we get the convergence of $\sum_{k=1}^{k_N} \left(\Delta F_k^N\right)^\alpha \left(\Delta_k^N\right)^{1-\alpha}$ to $\int_\real p^\alpha (x)\,dx>0$.
    
\end{remark}
\begin{proof}
$(i)$
Let $\alpha > 1$.
Then
\begin{equation}\label{eq:inequ1}
0 \le \Delta F_k^N \le \left(\int_{t_{k-1}^N}^{t_k^N} p^\alpha(x)\,dx\right)^{\frac1\alpha} \left(\Delta_k^N\right)^{1-\frac1\alpha}.
\end{equation}
Now, choose $\varepsilon > 0$.
There exists
$[a, b] \subset s(p) = \supp\set{p(x), x\in\real}$
such that
\[
\int_{[a,b]} p^\alpha(x)\,dx \in \left[(1+\varepsilon)^{-1} \int_\real p^\alpha(x)\,dx, \int_\real p^\alpha(x)\,dx\right].
\]
Then, using Lagrange theorem, we get that
\begin{align}
0 \le \delta_1^\varepsilon &= \log\left(\int_\real p^\alpha(x)\,dx\right) - \log\left(\int_{[a,b]} p^\alpha(x)\,dx\right)
\notag\\
&\le \frac{1}{\int_{[a,b]} p^\alpha(x)\,dx} \int_\real p^\alpha(x)\,dx \left(1 - \frac{1}{1 + \varepsilon}\right)
= \varepsilon.
\label{eq:inequ2}
\end{align}
Now taking into account continuity of $p$, we immediately get, similarly to \eqref{convint} that
\begin{align*}
\log S_1^N &\coloneqq \log \sum_{k: [t_{k-1}^N, t_k^N] \cap [a, b] \ne \emptyset}  \left(\Delta F_k^N\right)^\alpha \left(\Delta_k^N\right)^{1-\alpha} 
=  \log \sum_{k: [t_{k-1}^N, t_k^N] \cap [a, b] \ne \emptyset}  \left(p\left(\theta_k^N\right)\right)^\alpha \Delta_k^N
\\
&\to \log \int_a^b p^\alpha(x)\,dx,
\quad\text{as } N \to \infty,
\end{align*}
and therefore we can choose $N_0 \ge 1$ such that for all $N \ge N_0$
\begin{equation}\label{eq:inequ3}
\delta_2^N = \abs{\log S_1^N - \log \int_a^b p^\alpha(x)\,dx} < \varepsilon.
\end{equation}
Now, let us bound from above
\[
\delta_3^N = \abs{\log\sum_{k=1}^{k_N}  \left(\Delta F_k^N\right)^\alpha \left(\Delta_k^N\right)^{1-\alpha} - \log S_1^N}
= \abs{\log\left(1 + \frac{S_2^N}{S_1^N}\right)}
\le \frac{S_2^N}{S_1^N},
\]
where
\[
S_2^N = \sum_{k:  [t_{k-1}^N, t_k^N] \cap [a, b] = \emptyset}  \left(\Delta F_k^N\right)^\alpha \left(\Delta_k^N\right)^{1-\alpha}.
\]
Note that
\[
S_1^N \to \int_a^b p^\alpha(x)\,dx > 0.
\]
Now, let us bound $S_2^N$ with the help of \eqref{eq:inequ1}:
\[
\left(\Delta F_k^N\right)^\alpha \left(\Delta_k^N\right)^{1-\alpha}
\le \int_{t_{k-1}^N}^{t_k^N} p^\alpha(x)\,dx,
\]
whence
\[
S_2^N \le \int_{\real\setminus[a,b]} p^\alpha(x)\,dx
\le \frac{\varepsilon  \int_\real p^\alpha(x)\,dx}{1 + \varepsilon}.
\]
Finally,
\begin{equation}\label{eq:inequ4}
\frac{S_2^N}{S_1^N} \le \frac{\varepsilon \int_\real p^\alpha(x)\,dx}{(1 + \varepsilon) S_1^N}
\le \frac{\varepsilon}{1 + \varepsilon}\, \frac{\int_\real p^\alpha(x)\,dx}{\int_a^b p^\alpha(x)\,dx + \varepsilon}.
\end{equation}
Now, taking into account that
\[
\abs{\log\left(\int_\real p^\alpha(x)\,dx\right) - \log\left(\sum_{k=1}^{k_N} \left(\Delta F_k^N\right)^\alpha \left(\Delta_k^N\right)^{1-\alpha}\right)}
\le \delta_1^\varepsilon + \delta_2^N + \delta_3^N,
\]
inequalities \eqref{eq:inequ2}, \eqref{eq:inequ3} and  \eqref{eq:inequ4}, and arbitrary choice of $\varepsilon > 0$, we get the proof of $(i)$.

$(ii)$
Note that we applied the fact that $\alpha > 1$, only bounding $\delta_3^N$.
So, let now $0 < \alpha < 1$, and let us construct an upper bound for $\delta_3^N$. 
In fact, it means that we need to construct the upper bound for $S_2^N$.
Without loss of generality we can assume that $[-z, z] \subset [a, b]$.
Then, in particular, $p$ increases on any interval $[t_{k-1}^N, t_k^N]$ such that $[t_{k-1}^N, t_k^N] \cap [a, b] = \emptyset$, and $t_{k-1}^N > b$.
Consider only this case since the case where $t_k^N < a$ is considered similarly.
Denote $t_{k-1,1}^N$ the left endpoint of the first interval  $[t_{k-1}^N, t_k^N]$ that does not intersect with $[a, b]$ and $t_{k-1}^N > b$.
Then
\begin{align*}
S_{2,1}^N &\coloneqq \sum_{k:\, t_{k-1}^N \ge t_{k-1,1}^N} \left(\Delta F_k^N\right)^\alpha \left(\Delta_k^N\right)^{1-\alpha}
\le \sum_{k:\, t_{k-1}^N \ge t_{k-1,1}^N} p^\alpha \left(t_{k-1}^N\right) \Delta_k^N
\\
&= \sum_{k:\, t_{k-1}^N \ge t_{k-1,1}^N} p^\alpha\left(t_{k-1}^N\right) \frac1N
\le p^\alpha\left(t_{k-1,1}^N\right) \frac1N
+ \sum_{k:\, t_{k-1}^N > t_{k-1,1}^N} p^\alpha\left(t_{k-1}^N\right)\frac1N.
\end{align*}
But for any $t_{k-1}^N > t_{k-1,1}^N$
\[
p^\alpha\left(t_{k-1}^N\right)\frac1N
\le \int_{\frac{k-2}{N}}^{\frac{k-1}{N}} p^\alpha(x) \,dx,
\]
therefore
\[
S_{2,1}^N \le p^\alpha \left(t_{k-1,1}^N\right) \frac1N
+ \int_{t_{k-1,1}^N}^\infty p^\alpha(x) \,dx
\to 0 \quad \text{as } N \to \infty,
\]
and the proof of $(ii)$ follows.
\end{proof}

\subsection{Alternative versions of R\'{e}nyi entropy and their properties}
Now, let us construct the alternatives to R\'{e}nyi entropy, similar to \eqref{alterShann}--\eqref{alterShann3}:
\begin{gather*}
\entropy_{R,\alpha}^{(1)}(\{p(x), x \in \real\}) = \frac{1}{|1-\alpha|} \left|\log \left(\int_\real p^\alpha(x) \,dx\right)\right|;
\\
\entropy_{R,\alpha}^{(2)}(\{p(x), x \in \real\}) = \frac{1}{|1-\alpha|} \left(\log \left(\int_\real p^\alpha(x) \,dx\right)\right)_+;
\\
\entropy_{R,\alpha}^{(3)}(\{p(x), x \in \real\}) = \frac{1}{|1-\alpha|}   \log \left(\int_\real p^\alpha(x) \,dx +1\right) .
\end{gather*}
Of course, all of them are strictly positive. Moreover, under assumptions of Theorem \ref{theor3}, according to Remark \ref{imprem}, $\sum_{k=1}^{k_N} \left(\Delta F_k^N\right)^\alpha \left(\Delta_k^N\right)^{1-\alpha}\to \int_\real p^\alpha (x)\,dx>0$, and we  get that all alternative entropies are the limits of respective discrete functionals. 

\subsection{Gaussian distribution}
Let us now consider the case where $p_0(x)$ is the density of a centered normal distribution, as defined in equation~\eqref{eq:norm-pdf}.

\begin{proposition}
Let $\alpha > 0$, $\alpha \ne 1$, and define $\sigma_\alpha \coloneqq (2\pi)^{-1/2} \alpha^{1/[2(1-\alpha)]}$. Then
\begin{enumerate}[1)]
\item For any $\alpha \in (0,1) \cup (1,+\infty)$
\[
\entropy_{R,\alpha}^{(1)}(\{p_0(x), x \in \real\}) =
\begin{cases}
-\log \sigma - \dfrac{1}{2} \log(2\pi) + \dfrac{\log \alpha}{2(1-\alpha)}, & \text{if } \sigma \le \sigma_\alpha, \\
\log \sigma + \dfrac{1}{2} \log(2\pi) - \dfrac{\log \alpha}{2(1-\alpha)}, & \text{if } \sigma > \sigma_\alpha.
\end{cases}
\]
The function $\entropy_{R,\alpha}^{(1)}(\{p_0(x), x \in \real\})$ decreases from $+\infty$ to $0$ as $\sigma$ increases from $0$ to $\sigma_\alpha$, and increases from $0$ to $+\infty$ as $\sigma$ increases from $\sigma_\alpha$ to $+\infty$.

\item If $\alpha \in (0, 1)$, then
\[
\entropy_{R,\alpha}^{(2)}(\{p_0(x), x \in \real\}) =
\begin{cases}
0, & \text{if } \sigma \le \sigma_\alpha, \\
\log \sigma + \dfrac{1}{2} \log(2\pi) - \dfrac{\log \alpha}{2(1-\alpha)}, & \text{if } \sigma > \sigma_\alpha.
\end{cases}
\]
In this case, $\entropy_{R,\alpha}^{(2)}$ increases from $0$ to $+\infty$ as $\sigma$ increases from $\sigma_\alpha$ to $+\infty$.

If $\alpha > 1$, then
\[
\entropy_{R,\alpha}^{(2)}(\{p_0(x), x \in \real\}) =
\begin{cases}
-\log \sigma - \dfrac{1}{2} \log(2\pi) + \dfrac{\log \alpha}{2(1-\alpha)}, & \text{if } \sigma \le \sigma_\alpha, \\
0, & \text{if } \sigma > \sigma_\alpha.
\end{cases}
\]
In this case, $\entropy_{R,\alpha}^{(2)}$ decreases from $+\infty$ to $0$ as $\sigma$ increases from $0$ to $\sigma_\alpha$.

\item 
If $\alpha \in (0, 1)$, then $\entropy_{R,\alpha}^{(3)}(\{p_0(x), x \in \real\})$ increases from $0$ to $+\infty$ as $\sigma$ increases from $0$ to $+\infty$.

If $\alpha > 1$, then $\entropy_{R,\alpha}^{(3)}(\{p_0(x), x \in \real\})$ decreases from $+\infty$ to $0$ as $\sigma$ increases from $0$ to $+\infty$.
\end{enumerate}
\end{proposition}
\begin{proof}
As shown in \cite[formula (A1)]{entropy-gaussian}, we have
\[
\int_\real p_0^\alpha(x)\,dx = \sigma^{1-\alpha}(2\pi)^{(1-\alpha)/2} \alpha^{-1/2}.
\]
It follows that:
\begin{itemize}
\item If $\alpha < 1$, the integral increases from $0$ to $+\infty$ as $\sigma \to +\infty$;
\item If $\alpha > 1$, it decreases from $+\infty$ to $0$;
\item In both cases, the integral equals $1$ when $\sigma = \sigma_\alpha$.
\end{itemize}

The stated behavior of the entropies $\entropy_{R,\alpha}^{(i)}(\{p_0(x), x \in \real\})$, $i = 1, 2, 3$, then follows directly from their respective definitions.
\end{proof}

\begin{figure}
    \centering
\subfigure[$\alpha = 0.5$]{\includegraphics[scale=0.61]{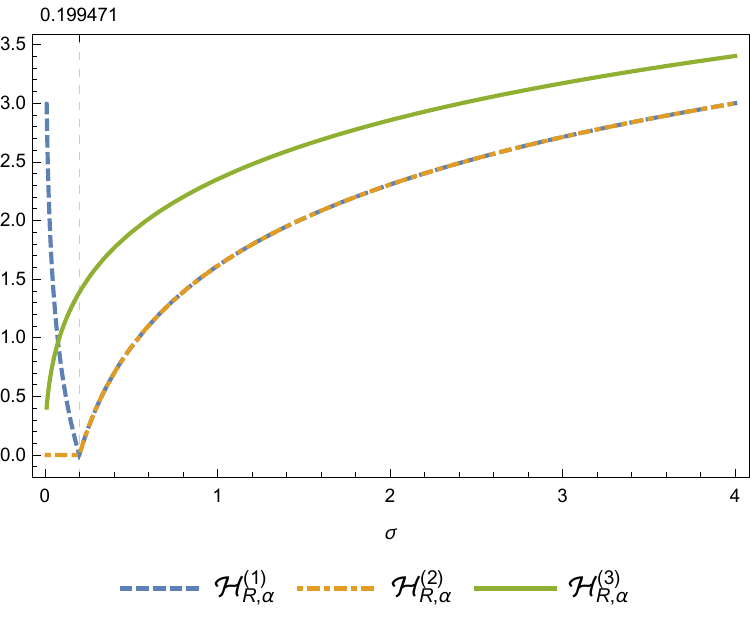}}\hfill
\subfigure[$\alpha = 1.5$]{\includegraphics[scale=0.61]{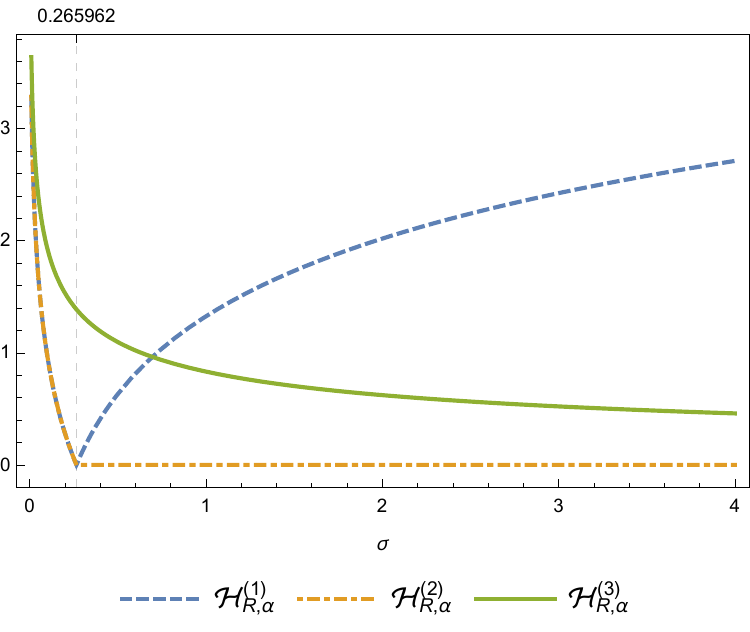}}
    \caption{$\entropy_{R,\alpha}^{(i)}$, $i = 1,2,3$, for Gaussian distribution as functions of $\sigma$ for $\alpha = 0.5$ and $\alpha = 1.5$}
\end{figure}

\subsection{Exponential distribution}
Consider exponential distribution with the density $p_1(x) = \mu^{-1} e^{- x/\mu}$, $x\ge 0$, $\mu > 0$ in the same spirit.
In this case, 
\[
\int_\real p_1^\alpha(x)\,dx = \frac{\mu^{1-\alpha}}{\alpha}
\]
(see \cite[Proposition 3.5]{BMR}), whence we get the following result.

\pagebreak
\begin{proposition}
Let $\alpha > 0$, $\alpha \ne 1$, and define $\mu_\alpha \coloneqq  \alpha^{1/(1-\alpha)}$. Then
\begin{enumerate}[1)]
\item For any $\alpha \in (0,1) \cup (1,+\infty)$
\[
\entropy_{R,\alpha}^{(1)}(\{p_1(x), x \in \real\}) =
\begin{cases}
-\log \mu + \dfrac{\log \alpha}{1-\alpha}, & \text{if } \mu \le \mu_\alpha, \\
\log \mu - \dfrac{\log \alpha}{1-\alpha}, & \text{if } \mu > \mu_\alpha.
\end{cases}
\]
The function $\entropy_{R,\alpha}^{(1)}(\{p_1(x), x \in \real\})$ decreases from $+\infty$ to $0$ as $\mu$ increases from $0$ to $\mu_\alpha$, and increases from $0$ to $+\infty$ as $\mu$ increases from $\mu_\alpha$ to $+\infty$.

\item If $\alpha \in (0, 1)$, then
\[
\entropy_{R,\alpha}^{(2)}(\{p_1(x), x \in \real\}) =
\begin{cases}
0, & \text{if } \mu \le \mu_\alpha, \\
\log \mu - \dfrac{\log \alpha}{1-\alpha}, & \text{if } \mu > \mu_\alpha.
\end{cases}
\]
In this case, $\entropy_{R,\alpha}^{(2)}$ increases from $0$ to $+\infty$ as $\mu$ increases from $\mu_\alpha$ to $+\infty$.

If $\alpha > 1$, then
\[
\entropy_{R,\alpha}^{(2)}(\{p_1(x), x \in \real\}) =
\begin{cases}
-\log \mu + \dfrac{\log \alpha}{1-\alpha}, & \text{if } \mu \le \mu_\alpha, \\
0, & \text{if } \mu > \mu_\alpha.
\end{cases}
\]
In this case, $\entropy_{R,\alpha}^{(2)}$ decreases from $+\infty$ to $0$ as $\mu$ increases from $0$ to $\mu_\alpha$.

\item For any $\alpha \in (0,1) \cup (1,+\infty)$
\[
\entropy_{R,\alpha}^{(3)}(\{p_1(x), x \in \real\}) =
\frac{1}{\abs{1 - \alpha}} \log \left(\mu^{1-\alpha}\alpha^{-1} + 1\right)
\]

If $\alpha \in (0, 1)$, then $\entropy_{R,\alpha}^{(3)}(\{p_1(x), x \in \real\})$ increases from $0$ to $+\infty$ as $\mu$ increases from $0$ to $+\infty$.

If $\alpha > 1$, then $\entropy_{R,\alpha}^{(3)}(\{p_1(x), x \in \real\})$ decreases from $+\infty$ to $0$ as $\mu$ increases from $0$ to~$+\infty$.
\end{enumerate}
\end{proposition}

\begin{figure}
    \centering
\subfigure[$\alpha = 0.5$]{\includegraphics[scale=0.61]{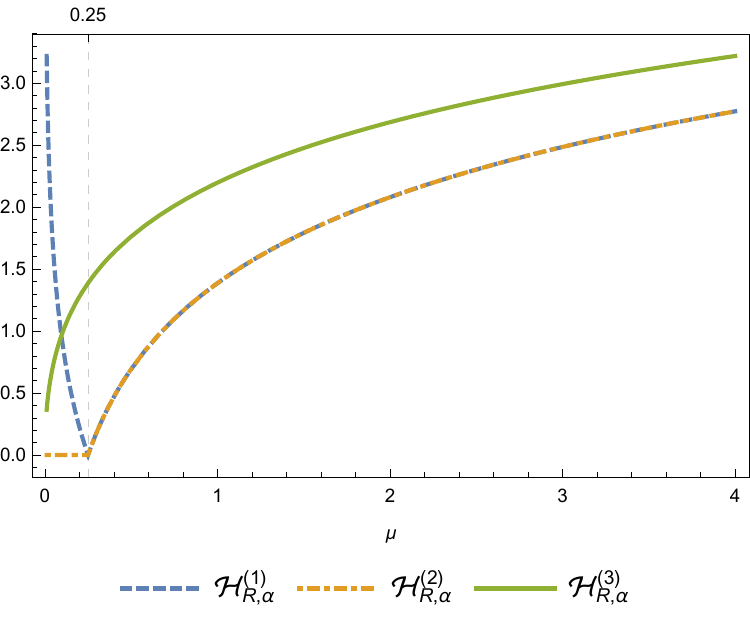}}\hfill
\subfigure[$\alpha = 1.5$]{\includegraphics[scale=0.61]{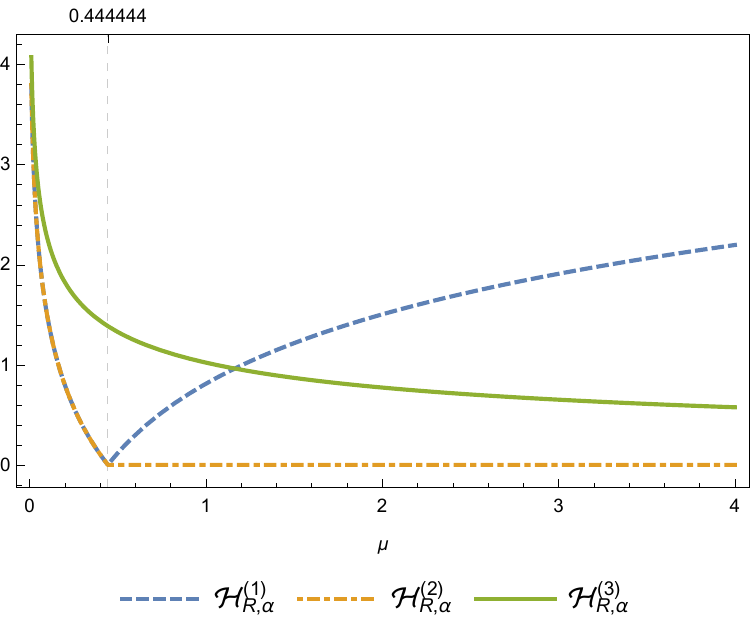}}
    \caption{$\entropy_{R,\alpha}^{(i)}$, $i = 1,2,3$, for exponential distribution as functions of $\mu$ for $\alpha = 0.5$ and $\alpha = 1.5$}
\end{figure}
\FloatBarrier


\begin{thebibliography}{10}

\bibitem{Aczel74}
J.~Acz\'el, B.~Forte, and C.~T. Ng.
\newblock Why the {S}hannon and {H}artley entropies are `natural'.
\newblock {\em Advances in Appl. Probability}, 6:131--146, 1974.

\bibitem{ali2023shannon}
A.~Ali, S.~Anam, and M.~M. Ahmed.
\newblock Shannon entropy in artificial intelligence and its applications based
  on information theory.
\newblock {\em J. Appl. Emerg. Sci}, 13(1):9--17, 2023.

\bibitem{BMR}
I.~Bodnarchuk, Y.~Mishura, and K.~Ralchenko.
\newblock Properties of the {S}hannon, {R}\'{e}nyi and other entropies:
  dependence in parameters, robustness in distributions and extremes, 2024.
\newblock https://arxiv.org/abs/2411.15817.

\bibitem{carcassi2021variability}
G.~Carcassi, C.~A. Aidala, and J.~Barbour.
\newblock Variability as a better characterization of {S}hannon entropy.
\newblock {\em European Journal of Physics}, 42(4):045102, 2021.

\bibitem{dwivedi2022application}
P.~P. Dwivedi and D.~K. Sharma.
\newblock Application of {S}hannon entropy and {CoCoSo} methods in selection of
  the most appropriate engineering sustainability components.
\newblock {\em Cleaner Materials}, 5:100118, 2022.

\bibitem{ellerman2021new}
D.~Ellerman.
\newblock {\em New foundations for information theory: logical entropy and
  {S}hannon entropy}.
\newblock Springer Nature, 2021.

\bibitem{Feutrill21}
A.~Feutrill and M.~Roughan.
\newblock A review of {S}hannon and differential entropy rate estimation.
\newblock {\em Entropy}, 23(8):Paper No. 1046, 19, 2021.

\bibitem{Jaynes62}
E.~T. Jaynes.
\newblock Information theory and statistical mechanics.
\newblock In {\em Statistical physics ({B}randeis {S}ummer {I}nstitute, 1962,
  {V}ol. 3)}, volume Vol. 3 of {\em Brandeis Summer Institute, 1962}, pages
  181--218. W. A. Benjamin, Inc., New York-Amsterdam, 1963.

\bibitem{entropy-gaussian}
A.~Malyarenko, Y.~Mishura, K.~Ralchenko, and Y.~A. Rudyk.
\newblock Properties of various entropies of {G}aussian distribution and
  comparison of entropies of fractional processes.
\newblock {\em Axioms}, 12(11):1026, 2023.

\bibitem{nascimento2024electron}
W.~S. Nascimento, A.~M. Maniero, F.~V. Prudente, C.~R. de~Carvalho, and
  G.~Jalbert.
\newblock Electron confinement study in a double quantum dot by means of
  {S}hannon entropy information.
\newblock {\em Physica B: Condensed Matter}, 677:415692, 2024.

\bibitem{Rodriguez22}
N.~Rodriguez-Rodriguez and O.~Miramontes.
\newblock Shannon entropy: an econophysical approach to cryptocurrency
  portfolios.
\newblock {\em Entropy}, 24(11):Paper No. 1583, 15, 2022.

\bibitem{Shannon48}
C.~E. Shannon.
\newblock A mathematical theory of communication.
\newblock {\em Bell System Tech. J.}, 27:379--423, 623--656, 1948.

\bibitem{strat}
R.~L. Stratonovich.
\newblock {\em Theory of information and its value}.
\newblock Belavkin, R. V., Pardalos, P. M. and Principe, J. C. (eds.),
  Springer, Cham, 2020.

\bibitem{zachary2021urban}
D.~Zachary and S.~Dobson.
\newblock Urban development and complexity: {S}hannon entropy as a measure of
  diversity.
\newblock {\em Planning Practice \& Research}, 36(2):157--173, 2021.

\end{thebibliography}
\end{document}